\documentclass[a4paper,11pt]{article}
\usepackage{amsmath, amssymb, amsthm}
\usepackage{pstricks, pst-node}

\newtheorem*{definition}{Definition}
\newtheorem*{convention}{Convention}
\newtheorem{theorem}{Theorem}
\newtheorem{lemma}[theorem]{Lemma}
\newtheorem{corollary}[theorem]{Corollary}

\title{Embeddings into the countable atomless Boolean algebra}
\author{Stijn Vermeeren}
\date{\today}

\begin{document}

\maketitle

\begin{abstract}
We prove that there is a lattice embedded from every countable distributive lattice into the Boolean algebra of computable subsets of $\mathbb{N}$. Along the way, we discuss all relevant results about lattices, Boolean algebras and embeddings between them.
\end{abstract}

The objective of this article is to prove that every countable distributive lattice can be embedded into the Boolean algebra of computable subsets of $\mathbb{N}$. I came across this fact in an article by Simpson \cite{simpson}, but he doesn't give a proof nor a reference. And Simpson is not the only one; I repeatedly found the result simply stated as a ``well-known fact''. An article by Ganchev and Soskova \cite{ganchev_soskova}, stating that one can use ``a compactness argument'' to prove the theorem, is the best I could find.

This article is the result of my investigation into the theorem and its proof. None of the results are new, but to my knowledge they have never been as coherently presented before. I define most concepts (such as \emph{lattice} and \emph{Boolean algebra}) from scratch, although very concisely. More background can be found in \cite{gratzer} or \cite{birkhoff}. I assume familiarity with the basic notions of model theory, which are explained e.g.\ in \cite{hodges}. I hope that this article makes the content, which was up to now hard to find and scattered between different resources, more accessible and understandable.

\section{Lattices}

\begin{definition}
A \textbf{lattice} $(L, \wedge, \vee)$ is a set nonempty $L$ with binary operations $\wedge$ (\emph{meet}) and $\vee$ (\emph{join}) that are commutative, associative, idempotent (i.e.\ $x \wedge x = x$ and $x \vee x = x$ for all $x \in L$) and satisfy two supplementary \emph{absorption} axioms:
\begin{align*}
(x \wedge y) \vee y &= y \\
(x \vee y) \wedge y &= y
\end{align*}
for all $x,y \in L$.\footnote{The absorption axioms actually imply idempotency. This is proved by simplifying $((x \wedge x) \vee x) \wedge x$ and $((x \vee x) \wedge x) \vee x$ in two ways. Still, idempotency is ussually included in the axioms.}
\end{definition}

Alternatively, a lattice can also be defined as a partially ordered set $(L, \le)$ such that each two elements $x,y \in L$ have a greatest lower bound (\emph{meet}) $x \wedge y$ and a least upper bound (\emph{join}) $x \vee y$. This obviously gives a structure $(L, \wedge, \vee)$ that satisfies the above definition of lattice. Conversely, if for elements $x,y$ of a lattice $(L, \wedge, \vee)$ we define $x \le y$ if and only if $x = x \wedge y$, then this gives a partial order with greatest lower bounds given by $\wedge$ and least upper bounds given by $\vee$, as is easily proved from the axioms.

Viewing a lattice as a partial order with binary meets and joins is often helpful in visualising the lattice, especially when it is finite. For example, the lattice of divisors of $12$, where $\wedge$ is greatest common divisor and $\vee$ is least common multiple, is most clearly represented when we consider the divisors to be ordered by divisibility and draw a diagram as follows:
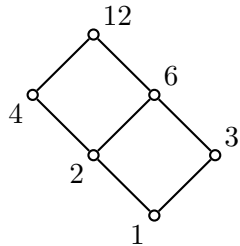
\begin{center}
	\psset{unit=0.8cm}
\begin{pspicture}(0,0.7)(9,4.5)
%\psgrid[subgriddiv=0,griddots=6,gridcolor=red]

{
\cnode(5,1){2.3pt}{n1}
\uput[dl](5,1){$1$}

\cnode(4,2){2.3pt}{n2}
\uput[dl](4,2){$2$}

\cnode(3,3){2.3pt}{n4}
\uput[dl](3,3){$4$}

\cnode(6,2){2.3pt}{n3}
\uput[ur](6,2){$3$}

\cnode(5,3){2.3pt}{n6}
\uput[ur](5,3){$6$}

\cnode(4,4){2.3pt}{n12}
\uput[ur](4,4){$12$}
}

\ncline{n1}{n2}
\ncline{n1}{n3}
\ncline{n2}{n6}
\ncline{n2}{n4}
\ncline{n3}{n6}
\ncline{n4}{n12}
\ncline{n6}{n12}

\end{pspicture}

\end{center}
Gowing up along the line segments in the drawing corresponds to going to greater elements in the order. The binary meets and joins can also be read directly from the diagram.

However, we want a homomorphism $f: L \to K$ of lattices to preserve meets and joins:
\begin{align*}
f(x \wedge y) &= f(x) \wedge f(y) \\
f(x \vee y) &= f(x) \vee f(y)
\end{align*}
for all $x,y \in L$. Hence we need to include $\wedge$ and $\vee$ as primitive function symbols in the first order language of lattice theory.\footnote{We must be careful to distinguish the different possible meanings of $\wedge$ (\emph{conjunction} or \emph{meet}) and $\vee$ (\emph{disjunction} or \emph{join}). When formally defining the first order language of lattice theory, we must choose different symbols for the different meanings. However in this article, we won't need any conjunction or disjunctions symbols, so $\wedge$ and $\vee$ will always mean respectively \emph{meet} and \emph{join}.} If we would include only $\le$ as function symbol and ensure the existence of joins and meets through axioms, then the usual definition of homomorphism (see e.g.\ \cite{hodges}) would include a map like
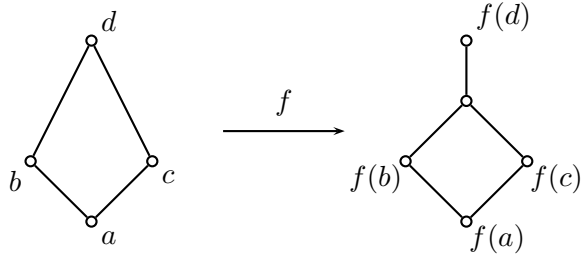
\begin{center}
	\psset{unit=0.8cm}
\begin{pspicture}(0,0.7)(10,4.3)
%\psgrid[subgriddiv=0,griddots=6,gridcolor=red]

{
\cnode(2,1){2.3pt}{na}
\uput[dr](2,1){$a$}

\cnode(1,2){2.3pt}{nb}
\uput[dl](1,2){$b$}

\cnode(3,2){2.3pt}{nc}
\uput[dr](3,2){$c$}

\cnode(2,4){2.3pt}{nd}
\uput[ur](2,4){$d$}
}

\ncline{na}{nb}
\ncline{na}{nc}
\ncline{nb}{nd}
\ncline{nc}{nd}

{ 
\cnode(8,1){2.3pt}{nfa}
\uput{1.5pt}[dr](8,1){$f(a)$}

\cnode(7,2){2.3pt}{nfb}
\uput{1.8pt}[dl](7,2){$f(b)$}

\cnode(9,2){2.3pt}{nfc}
\uput{1.5pt}[dr](9,2){$f(c)$}

\cnode(8,3){2.3pt}{nx}

\cnode(8,4){2.3pt}{nfd}
\uput[ur](8,4){$f(d)$}
}

\ncline{nfa}{nfb}
\ncline{nfa}{nfc}
\ncline{nfb}{nx}
\ncline{nfc}{nx}
\ncline{nx}{nfd}

\pcline{->}(4,2.5)(6,2.5)
\aput{:U}{$f$}

\end{pspicture}

\end{center}
which preserves order but does not preserve the join of $b$ and $c$.

A lattice $(L, \wedge, \vee)$ is \textbf{distributive} if $\wedge$ distributes over $\vee$ and vice versa, i.e.
\begin{align*}
	(x \vee y) \wedge z &= (x \wedge z) \vee (y \wedge z) \\
	(x \wedge y) \vee z &= (x \vee z) \wedge (y \vee z)
\end{align*}
for all $x,y,z \in L$.

The smallest example of a \emph{non}-distributive lattice is the \textbf{diamond lattice} $M_3$:
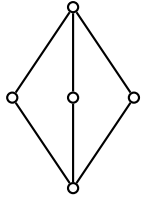
\begin{center}
	\psset{unit=0.8cm}
\begin{pspicture}(1,0.9)(9,4.1)
%\psgrid[subgriddiv=0,griddots=6,gridcolor=red]

{
\cnode(5,1){2.3pt}{n0}
\cnode(4,2.5){2.3pt}{na}
\cnode(5,2.5){2.3pt}{nb}
\cnode(6,2.5){2.3pt}{nc}
\cnode(5,4){2.3pt}{n1}
}

\ncline{n0}{na}
\ncline{n0}{nb}
\ncline{n0}{nc}
\ncline{na}{n1}
\ncline{nb}{n1}
\ncline{nc}{n1}

\end{pspicture}

\end{center}

A \textbf{lattice of sets} is a collection of sets which is closed under binary intersections (meets) and binary unions (joins). Any lattice of sets is distributive. In fact, every distributive lattice is isomorphic to a lattice of sets. We will prove this in the next section for finite distributive lattices. The infinite distributive lattices, a proof is given in the appendix.

\section{Birkhoff's representation theorem}

Birkhoff's representation theorem says that every finite distributive lattice is isomorphic to a lattice of sets. In order to prove this, we need to introduce a new concept:

\begin{definition}
Suppose we have a lattice $(L, \wedge, \vee)$. A non-minimal element $a \in L$ is
\begin{itemize}
  \setlength{\itemsep}{0pt}
\item \textbf{join-irreducible} if for all $b,c \in L$, when $a = b \vee c$, then $a = b$ or $a = c$;
\item \textbf{join-prime} if for all $b,c \in L$, when $a \le b \vee c$, then $a \le b$ or $a \le c$.
\end{itemize}
\end{definition}

In the non-distributive diamond lattice
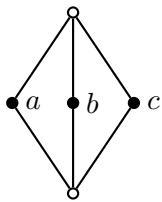
\begin{center}
	\psset{unit=0.8cm}
\begin{pspicture}(1,0.9)(9,4.1)
%\psgrid[subgriddiv=0,griddots=6,gridcolor=red]

{
\cnode(5,1){2.3pt}{n0}
\cnode[fillcolor=black,fillstyle=solid](4,2.5){2.3pt}{na}
\uput[r](4,2.5){$a$}
\cnode[fillcolor=black,fillstyle=solid](5,2.5){2.3pt}{nb}
\uput[r](5,2.5){$b$}
\cnode[fillcolor=black,fillstyle=solid](6,2.5){2.3pt}{nc}
\uput[r](6,2.5){$c$}
\cnode(5,4){2.3pt}{n1}
}

\ncline{n0}{na}
\ncline{n0}{nb}
\ncline{n0}{nc}
\ncline{na}{n1}
\ncline{nb}{n1}
\ncline{nc}{n1}

\end{pspicture}

\end{center}
the elements $a$, $b$ and $c$ are join-irreducible, but not join-prime. For example $a \le b \vee c$ but neither $a \le b$ nor $a \le c$.

In the lattice of divisors of $12$, the elements with exactly one prime divisor ($2$, $4$ and $3$) are both the join-irreducibles and the join-primes:
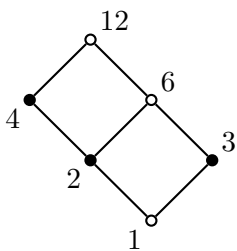
\begin{center}
	\psset{unit=0.8cm}
\begin{pspicture}(0,0.7)(9,4.5)
%\psgrid[subgriddiv=0,griddots=6,gridcolor=red]
{
\cnode(5,1){2.3pt}{n1}
\uput[dl](5,1){$1$}

\cnode[fillcolor=black,fillstyle=solid](4,2){2.3pt}{n2}
\uput[dl](4,2){$2$}

\cnode[fillcolor=black,fillstyle=solid](3,3){2.3pt}{n4}
\uput[dl](3,3){$4$}

\cnode[fillcolor=black,fillstyle=solid](6,2){2.3pt}{n3}
\uput[ur](6,2){$3$}

\cnode(5,3){2.3pt}{n6}
\uput[ur](5,3){$6$}

\cnode(4,4){2.3pt}{n12}
\uput[ur](4,4){$12$}
}

\ncline{n1}{n2}
\ncline{n1}{n3}
\ncline{n2}{n6}
\ncline{n2}{n4}
\ncline{n3}{n6}
\ncline{n4}{n12}
\ncline{n6}{n12}

\end{pspicture}

\end{center}

We can indeed prove that join-prime is stronger than join-irreducible in general and that in distributive lattices the two concepts are equivalent.

\begin{lemma}
\label{irr-prime}
\begin{itemize}
  \setlength{\itemsep}{0pt}
\item In a lattice $(L, \wedge, \vee)$, if $a \in L$ is join-prime, then it is join-irreducible.
\item If the lattice is distributive, then the converse holds as well.
\end{itemize}
\end{lemma}

\begin{proof}
\begin{itemize}
\item Suppose $a \in L$ is join-prime and can be written as 
\[ a = b \vee c \]
for some $b,c \in L$. Applying the join-primality of $a$ gives $a \le b$ or $a \le c$. Suppose without loss of generality that $a \le b$. But we have $b \le a$ as well, because $a$ is the join of $b$ and $c$. Hence $a = b$ and $a$ is join-irreducible.
\item Suppose that $L$ is distributive, that $a \in L$ is join-irreducible and that $a \le b \vee c$ for some $b,c \in L$. So
\[ a = a \wedge (b \vee c) = (a \wedge b) \vee (a \wedge c). \]
As $a$ is join-irreducible, one of $a \wedge b$ and $a \wedge c$ must be equal to $a$. Suppose without loss of generality that $a = a \wedge b$. This means $a \le b$. Hence $a$ is join-prime.
\end{itemize}
\end{proof}

\begin{convention}
In a lattice with a least element $0$, we take an empty join to be equal to $0$:
\[ \bigvee_{p \in \emptyset} p = 0. \]
This applies in particular to finite lattices and Boolean algebras.
\end{convention}

We are now ready to prove \textbf{Birkhoff's representation theorem}:

\begin{theorem} \emph{(Birkhoff 1933, \cite{birkhoffrepr})}
A finite distributive lattice $(L, \wedge, \vee)$ is isomorphic to the lattice of lower (i.e.\ downward closed) sets of join-prime elements.
\end{theorem}

\begin{proof}
We claim that the embedding 
\[ f: a \mapsto \{ p \le a \ :\ \mbox{$p$ is join-prime}\} \]
works. It certainly preserves meets, because meet corresponds to greatest lower bound. It preserves joins as well, by definition of join-prime.

To show \textbf{surjectivity}, let $P$ be a downward closed set of join-prime elements. Consider $a = \displaystyle\bigvee_{p \in P} p$. Certainly $P \subseteq f(a)$, so we still need to prove that $f(a) \subseteq P$. This certainly holds if $P = \emptyset$ and thus $a = 0$. Otherwise, suppose that $q \in f(a)$, so $q$ is join-prime and 
\[ q \le a = \displaystyle\bigvee_{p \in P} p.\]
Applying $|P| - 1$ times the join-primality of $q$, gives us $q \le p$ for some $p \in P$. But $P$ is downward closed, so also $q \in P$, as required.

To show \textbf{injectivity}, we claim that $a = \displaystyle\bigvee_{p \in f(a)} p$ for all $a \in L$. (If so, then $f(a) = f(b)$ immediately implies $a = b$.) We prove this by induction. The claim is certainly valid voor $a = 0$. So take a nonzero $a \in L$ and assume that we have already proven that $b = \displaystyle\bigvee_{p \in f(b)} p$ for the finitely many $b \in L$ with $b < a$. Certainly we have $p \le a$ for all $p \in f(a)$, so also $\displaystyle\bigvee_{p \in f(a)} p \le a$. If $a$ is join-irreducible, then it is join-prime as well by Lemma \ref{irr-prime} and we have $a \in f(a)$, so $\displaystyle\bigvee_{p \in f(a)} p = a$. Otherwise $a$ is not join-irreducible, say $a = b \vee c$ where $b,c \in L$ with $b < a$ and $c < a$. By the induction hypothesis, $b$ and $c$ are the join of all join-primes below them. Futhermore every join-prime below $a$ is also below one of $b$ and $c$, by definition of join-prime. So 
\[ a = b \vee c = \left( \bigvee_{p \in f(b)} p \right) \vee  \left( \bigvee_{p \in f(c)} p \right) = \bigvee_{p \in f(a)} p \]
as required.
\end{proof}

Note that distributivity is only used when proving injectivity.

For the lattice of divisors of $12$, the isomorphism is as follows:
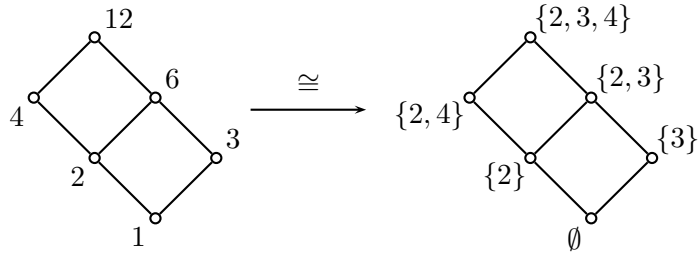
\begin{center}
	\psset{unit=0.8cm}
\begin{pspicture}(0,0.7)(12,4.5)
%\psgrid[subgriddiv=0,griddots=6,gridcolor=red]

{
\cnode(3,1){2.3pt}{n1}
\uput[dl](3,1){$1$}

\cnode(2,2){2.3pt}{n2}
\uput[dl](2,2){$2$}

\cnode(1,3){2.3pt}{n4}
\uput[dl](1,3){$4$}

\cnode(4,2){2.3pt}{n3}
\uput[ur](4,2){$3$}

\cnode(3,3){2.3pt}{n6}
\uput[ur](3,3){$6$}

\cnode(2,4){2.3pt}{n12}
\uput[ur](2,4){$12$}
}

\ncline{n1}{n2}
\ncline{n1}{n3}
\ncline{n2}{n6}
\ncline{n2}{n4}
\ncline{n3}{n6}
\ncline{n4}{n12}
\ncline{n6}{n12}

{
\cnode(10,1){2.3pt}{m1}
\uput[dl](10,1){$\emptyset$}

\cnode(9,2){2.3pt}{m2}
\uput{2pt}[dl](9,2){$\{2\}$}

\cnode(8,3){2.3pt}{m4}
\uput{2pt}[dl](8,3){$\{2, 4\}$}

\cnode(11,2){2.3pt}{m3}
\uput{2pt}[ur](11,2){$\{3\}$}

\cnode(10,3){2.3pt}{m6}
\uput{2pt}[ur](10,3){$\{2, 3\}$}

\cnode(9,4){2.3pt}{m12}
\uput{2pt}[ur](9,4){$\{2, 3, 4\}$}
}

\ncline{m1}{m2}
\ncline{m1}{m3}
\ncline{m2}{m6}
\ncline{m2}{m4}
\ncline{m3}{m6}
\ncline{m4}{m12}
\ncline{m6}{m12}

\pcline{->}(4.4,2.8)(6.3,2.8)
\aput{:U}{$\cong$}

\end{pspicture}

\end{center}

As mentioned before, Birkhoff's representation theorem can be generalized. Any infinite distributive lattice is isomorphic to a lattice of sets as well. However it is not possible to prove this using join-prime elements like in the finite case. And infinite distributive lattice might not contain any join-prime or even join-irreducible elements. For example the natural numbers, ordered inversely from large to small, form a distributive lattice where $\wedge$ is least common multiple and $\vee$ is greatest common divisor. However there are no join-irreducible elements, because every $n \in \mathbb{N}$ is the greatest common divisor of e.g. $2n$ and $3n$.

At the end of the article we will be able prove that every infinite distributive lattice is isomorphic to a lattice of sets, as a corollary of Birkhoff's representation theorem, using the compactness theorem and Stone's representation theorem.

\section{Boolean algebras}

\begin{definition}
A \textbf{Boolean algebra} (\emph{BA}) $(A, \wedge, \vee, 0, 1, \neg)$ is a distributive lattice $(A, \wedge, \vee)$ with least element $0$ and greatest element $1$ and with a unary operator $\neg$ (\emph{complementation}) that satisfies
\begin{align*}
	\neg x \wedge x &= 0 \\
	\neg x \vee x &= 1
\end{align*}
for all $x \in A$.
\end{definition}

The most common example of a BA is the \textbf{power set algebra} of a set $X$. The elements of this power set algebra are the subsets of $X$, $\wedge$ is intersection, $\vee$ is union, $0$ is the empty set, $1$ is the whole of $X$ and $\neg$ is complementation in $X$.

A subalgebra of a power set algebra (i.e.\ a collection of subsets of $X$ which contains $\emptyset$ and is closed under binary unions, binary intersections and complementation in $X$) is called an \textbf{algebra of sets}.

An \textbf{atom} in a partial order $(P, \le)$ with least element $0$ is a nonzero element $a \in P$ such that there is no $x \in P$ with $0 < x < a$. Every power set algebra has atoms, namely its singletons. 

We prove that every finite BA is isomorphic to a power set algebra, namely to the power set algebra of its set of atoms. The proof is inspired by \cite{goodstein}.

\begin{theorem}
A finite BA $(A, \wedge, \vee, 0, 1, \neg)$ is isomorphic to the power set algebra of the set of its atoms.
\end{theorem}

\begin{proof}
We claim that the isomorphism
\[ f: a \mapsto \{ p \le a \ :\ \mbox{$p$ is an atom}\} \]
works. It is trivial to verify that $f$ preserves meets, $0$ and $1$. To check that $f$ preserves joins, note that if for an atom $p$ we have $p \le a \vee b$, then 
\[ p =  p \wedge (a \vee b) = (p \wedge a) \vee (p \wedge b). \]
As $p$ is an atom, both $p \wedge a$ and $p \wedge b$ are either $0$ or $p$, and at least one of both must by $p$ so $p \le a$ or $p \le b$. Similarly, we can check that $f$ preserves complementation by considering
\[ p = p \wedge (a \vee \neg a) = (p \wedge a) \vee (p \wedge \neg a). \]
This shows that an atom $p$ is below at least one of $a$ and $\neg a$. Moreover the only element below $a$ and $\neg a$ is $0$, which is not an atom.

To prove \textbf{surjectivity}, let $P$ be a set of atoms. If $P$ is empty, then $f(0) = P$. Otherwise, we claim that $a = \displaystyle\bigvee_{p \in P} p$ is an element of $A$ with $f(a) = P$. Every atom in $P$ is certainly below $a$, so it remains to prove that $f(a) \subseteq P$. Take an atom $q \in f(a)$, that is $q \le \displaystyle\bigvee_{p \in P} p$. So
\[ q = q \wedge \displaystyle\bigvee_{p \in P} p = \displaystyle\bigvee_{p \in P} (q \wedge p). \]
As $q$ is an atom, $q = q \wedge p$ for some $p \in P$. But this $p$ is an atom as well, so $p = q$, and $q \in P$ as required.

To prove \textbf{injectivity}, we claim that $a = \displaystyle\bigvee_{p \in f(a)} p$ for every $a \in A$. Then $f(a) = f(b)$ would immediately imply that $a = b$. Certainly any $p \in f(a)$ is below $a$, so we only need to prove that $a \le \displaystyle\bigvee_{p \in f(a)} p$. Suppose for contradiction that $a \not\le \displaystyle\bigvee_{p \in f(a)} p$. Then $\neg a \wedge \displaystyle\bigvee_{p \in f(a)} p$ is nonzero, so (because the BA is finite) we can certainly find an atom $q$ below it. But then, like before, we get $q \in f(a)$ so $q \le a$, which contradicts $q \le \neg a$. Hence $a = \displaystyle\bigvee_{p \in f(a)} p$, as required.
\end{proof}

\begin{corollary}
Every finite BA has cardinality $2^n$ for some $n \in \mathbb{N}$.
\end{corollary}

\begin{proof}
By the previous theorem, every finite BA has the cardinality of some power set.
\end{proof}

\begin{corollary}
In every finite $BA$, each nonzero element can be uniquely written as the join of some atoms.
\end{corollary}

\begin{corollary}
Given two BAs with equally many atoms, then any bijection between their sets of atoms induces a unique isomorphism between the two BAs.
\end{corollary}

\section{The countable atomless BA}

Consider the BA of subsets of $\mathbb{N}$. This is isomorphic to the BA of $0,1$-sequences (i.e.\ elements of $\{0,1\}^\mathbb{N}$ where $\wedge$ is pointwise minimum and $\vee$ is pointwise maximum) by the isomorphism sending a subset of $\mathbb{N}$ to its characteristic function.

If we have a finite string of $0$'s and $1$'s (e.g.\ $01001$) then we can repeat this string infinitely often to obtain an element of $\{0,1\}^\mathbb{N}$, which we write as
\begin{align*}
	\overline{01001} &= 01001\,01001\,01001\ldots \\ 
		&= \{n \in \mathbb{N} \ \ :\ \ n \equiv 1,4 \pmod 5 \}
\end{align*}
Like this we can obtain all periodic subsets of $\mathbb{N}$, i.e.\ all $0,1$-sequences $a$ such that there is a nonzero $k \in \mathbb{N}$ such that $a(n) = a(n+k)$ for all $n \in \mathbb{N}$. These periodic sets form a subalgebra of the BA of subsets of $\mathbb{N}$. This subalgebra is different in two important ways from every power set algebra: it is countably infinite (trivially) and it is atomless.

Recall that every power set algebra has atoms, namely its singletons. The BA of periodic subsets of $\mathbb{N}$ however has no atoms. Indeed from a nonzero element $\overline{x_0x_1\ldots x_{k-1}}$ we can always get \emph{closer to $0$} by adding zeroes until the length of the period is doubled:
\[ 0 < \overline{x_0x_1\ldots x_{k-1}\underbrace{00\ldots 0}_{\mbox{$k$ zeroes}}} < \overline{x_0x_1\ldots x_{k-1}} \]

We will call the BA of periodic subsets of $\mathbb{N}$ \textbf{the countable atomless Boolean algebra} (\emph{CABA}). Indeed there is only one countable atomless BA up to isomorphism. We will proof this using a back-and-forth argument. (For more on back-and-forth proofs, see \cite{hodges}.)

\begin{theorem}
Any two countable atomless Boolean algebras $(A, \wedge, \vee, 0, 1, \neg)$ and $(B, \wedge, \vee, 0, 1, \neg)$ are isomorphic.
\end{theorem}
\begin{proof}
We construct successively bigger finite subalgebras
\[ \{0,1\} = A_0 \subset A_1 \subset A_2 \subset \ldots  \]
of $A$ and 
\[ \{0,1\} = B_0 \subset B_1 \subset B_2 \subset \ldots  \]
of $B$ such that
\[ \bigcup_{i \in \mathbb{N}} A_i = A \quad \mbox{and} \quad  \bigcup_{i \in \mathbb{N}} B_i = B, \]
and isomorphisms $f_0: A_0 \to B_0,\ f_1: A_1 \to B_1,\ f_2: A_2 \to B_2,\ \ldots$ such that
\[ f_0 \subset f_1 \subset f_2 \subset \ldots. \]
Then $f = \bigcup_{i \in \mathbb{N}} f_i$ will be the required isomorphism $A \cong B$.

List the elements of $A$ and $B$ as
\begin{align*}
	A &= \{a_0, a_1, a_2, a_3, \ldots \} \\
	B &= \{b_0, b_1, b_2, b_3, \ldots \}
\end{align*}

It is trivial to construct $f_0$.

So suppose that $n$ is even and we have constructed $A_n$, $B_n$ and $f_n$. As $A_n$ is generated by its atoms $p_0, p_1, \ldots, p_k$. By the isomorphism $f_n$, $B_n$ has atoms $f_n(p_0), f_n(p_1), \ldots, f_n(p_k)$. Let $x$ be the first element in $a_0, a_1, a_2, a_3, \ldots$ that is not in $A_n$. Let $A_{n+1}$ be the subalgebra of $A$ generated by $A_n$ and $x$. By using the disjunctive normal form, every element of $A_{n+1}$ can be written as the join of the elements in a subset of
\[ X = \{ p_0 \wedge x, p_0 \wedge \neg x, p_1 \wedge x, p_1 \wedge \neg x, \ldots, p_k \wedge x, p_k \wedge \neg x \}. \]
Hence $A_{n+1}$ is finite, and its atoms are the nonzero elements of $X$.

Now, for $i=0,\ldots,k$, define $x_i = p_i \wedge x$. Next, pick an element $y_i \in B$ such that
\[\begin{array}{ll}
	y_i = 0 & \mbox{if $x_i = 0$} \\
	y_i = f_n(p_i) & \mbox{if $x_i = p_i$} \\ 
	0 < y_i < f_n(p_i) & \mbox{if $0 < x_i < p_i$}
\end{array}\]
which is always possible because $B$ is atomless. Define $y = \bigvee_{i = 0}^k y_i$ and let $B_{n+1}$ be the subalgebra of $B$ generated by $B_n$ and $y$. Like before, every element of $B_{n+1}$ can be written as the join of the elements in a subset of
\[ Y = \{ f_n(p_0) \wedge y, f_n(p_0) \wedge \neg y, \ldots, f_n(p_k) \wedge f_n(x), f_n(p_k) \wedge \neg f_n(x) \}. \]
Hence $B_{n+1}$ is finite and its atoms are the nonzero elements of $Y$. Hence we can define $f_{n+1}$ to extend $f_n$ by mapping $x$ to $y$. Indeed, this induces a bijection of the atoms of $A_{n+1}$ to the atoms of $B_{n+1}$, so it gives us a well-defined isomorphism.

If $n$ is odd and then we construct $A_{n+1}$, $B_{n+1}$ and $f_{n+1}$ similarly, but with $A$ and $B$ switched around. (E.g., we let $x$ be the first element in $b_0, b_1, b_2, b_3, \ldots$ that is not in $B_n$, etc.) This makes sure that
\[ \bigcup_{i \in \mathbb{N}} B_i = B \quad \mbox{as well as} \quad  \bigcup_{i \in \mathbb{N}} A_i = A. \]
\end{proof}

\section{Embeddings into the countable atomless BA}

\begin{theorem}
Every finite Boolean algebra can be embedded into the countable atomless Boolean algebra.
\end{theorem}

\begin{proof}
Because every finite Boolean algebra is isomorphic to a power set algebra, it is sufficient to prove this for the power set algebra of subsets of $\{0,1,\ldots,k\}$. For $X \subseteq \{0,1,\ldots,k\}$, let $\chi_X$ be its characteristic function. Consider the map
\begin{align*}
\mathcal(P) &\to CABA \\
X &\mapsto \overline{\chi_X(0)\chi_X(1)\ldots \chi_X(k)}.
\end{align*}
This is obviously an injective homomorphism of Boolean algebras, as required.
\end{proof}

We can now find a lattice embedding from any finite distributive lattice into a finite Boolean algebra, and we have a Boolean algebra embedding of the latter into the countable atomless Boolean algebra. Hence:

\begin{corollary}
Every finite distributive lattice can be lattice embedded into the countable atomless Boolean algebra.
\end{corollary}

We can extend this result to countable distributive lattices, using the compactness theorem.

\begin{theorem}
\label{mainthm}
Every countable distributive lattice $(L, \wedge, \vee)$ can be lattice embedded into the countable atomless Boolean algebra.
\end{theorem}

\begin{proof}
Expand the first order language of Boolean algebra to include a constant symbol $\underline{a}$ for every element $a \in L$. Then consider the theory $T$ that consists of:
\begin{itemize}
  \setlength{\itemsep}{0pt}
\item the axioms of an atomless Boolean algebra,
\item $\underline{a} \not= \underline{b}$ for each distinct $a,b \in L$,
\item $\underline{a} \wedge \underline{b} = \underline{a \wedge b} $ for all $a,b \in L$,
\item $\underline{a} \vee \underline{b} = \underline{a \vee b} $ for all $a,b \in L$.
\end{itemize}
A model for this theory is then an atomless Boolean algebra with a lattice embedding of $L$ into it, given by the asignment of the constants $\underline{a}$ for $a \in L$. Every finite subset $T'$ of $T$ has a model. Indeed $T'$ involves only finitely many constants $\underline{a}$ and we only need to embed the lattice generated by the corresponding elements of $L$ into an atomless Boolean algebra. But this lattice is finite and distributive, so it is possible by the preceding corollary.

By the compactness theorem, there is a model for $T$. Indeed, because $T$ is countable, we can take this model to be countable as well. So we have a countable atomless Boolean algebra with an embedding of $L$ into it, as required.
\end{proof}

\begin{corollary}
Every countable distributive lattice can be embedded into the BA of computable subsets of $\mathbb{N}$.
\end{corollary}

\begin{proof}
By the proposition above, every countable distributive lattice can be embedded into the countable atomless BA of periodic subsets of $\mathbb{N}$, which is a subalgebra of the BA of computable sets.
\end{proof}

\section{Appendix: Stone's Representation Theorem}

Stone's representation theorem, proven by Marshall Harvey Stone in 1936, says that every BA is isomorphic to an algebra of sets. In order to prove this, we need to introduce the notions of filters and ultrafilters in a BA. This is a straightforward generalisation of the more commonly known concepts of filters and ultrafilters on a set. Indeed an (ultra)filter on a set $X$ will be exactly an (ultra)filter in the power set algebra of $X$, and all relevant results about filters on sets will still be valid for BAs.

\begin{definition}
\begin{itemize}
\item A \textbf{filter} $F$ in a BA $(A, \wedge, \vee, 0, 1, \neg)$ is a proper subset $F \subset A$ which is upwards closed and closed under binary meets. That is:
\[ a \wedge b \in F \]
for all $a,b \in F$ and
\[ b \in F \]
whenever $a \in F$ and $a \le b$.
\item If we order filters by set inclusion, then a filter which is maximal for this order is called an \textbf{ultrafilter}. $Ultra(A)$ is the set of all ultrafilters of the BA $(A, \wedge, \vee, 0, 1, \neg)$.
\end{itemize}
\end{definition}

Any subset $B$ of a BA $(A, \wedge, \vee, 0, 1, \neg)$ such that finite meets of elements in $B$ are nonzero, \emph{generates} a filter, namely the filter of all elements of $A$ which are greater or equal than some finite meet of elements in $B$. This is the smallest filter containing all elements in $B$.

A straighforward application or Zorn's lemma proves that every filter can be extended to an ultrafilter.

Recall that an \emph{atom} in a partial order $(P, \le)$ with least element $0$ is an element $a \in P$ such that there is no $x \in P$ with $0 < x < a$. Every power set algebra has atoms, namely its singletons. If $a$ is an atom in the BA $(A, \wedge, \vee, 0, 1, \neg)$, then the filter generated by $\{a\}$ is an ultrafilter. In a finite BA, all the ultrafilters are generated by an atom. However, in an infinite BA, there are many other ultrafilters and they are very hard to visualize. Luckily we can still proof some useful lemmas about them.

\begin{lemma}
A filter $F$ in a BA $(A, \wedge, \vee, 0, 1, \neg)$ is an ultrafilter is and only if for each $a \in A$, exactly one of $a$ and $\neg a$ is in $F$.
\end{lemma}

\begin{proof}
A filter $F$ with exactly one of $a$ and $\neg a$ is in $F$ for every $a \in A$ is certainly maximal and thus an ultrafilter, as no filter can contain both $a$ and $\neg a$.

Conversely, suppose that neither of $a$ and $\neg a$ is in $F$. Then $a$ has nonzero meet with every $f \in F$, because otherwise we would have $f \le \neg a$ so $\neg a \in F$. Hence $F \cup \{a\}$ has nonzero finite meets and generates a filter which strictly contains $F$. So $F$ is not an ultrafilter.
\end{proof}

We will not need the full strength of the following lemma, but prove it anyway for the reader's reference.

\begin{lemma}
Suppose we have an ultrafilter $U$ in a BA $(A, \wedge, \vee, 0, 1, \neg)$ and we have elements $a_1, \ldots, a_n \in A$ and $b \in U$. If $b \leq a_1 \vee \ldots \vee a_n$, then at least one of $a_1, \ldots, a_n$ is also in $U$.
\end{lemma}

\begin{proof}
Suppose for contradiction that none of $a_1, \ldots, a_n$ are in $U$. Then by the previous lemma all of $\neg a_1, \ldots, \neg a_n$ are in $U$, as well as
\[ \neg a_1 \wedge \ldots \wedge \neg a_n \in U. \]
However 
\[ \neg \left( \neg a_1 \wedge \ldots \wedge \neg a_n \right) = a_1 \vee \ldots \vee a_n \]
is greater than $b \in U$ and hence must be in $U$ as well, a contradiction.
\end{proof}

We are now ready to prove \textbf{Stone's representation theorem}.

\begin{theorem} \emph{(Stone 1936, \cite{stonerepr})}
Every BA $(A, \wedge, \vee, 0, 1, \neg)$ is isomorphic to an algebra of sets. Indeed there is a algebra-embedding $s$ from $A$ into the power set algebra of $Ultra(A)$ given by
\[ s(a) = \{U \in Ultra(A)\ \ :\ \ a \in U\}. \]
\end{theorem}

\begin{proof}
We have to verify that $s$ is indeed an embedding of BAs. We have
\[ s(0) = \emptyset \]
as no filter can contain $0$. Similarly
\[ s(1) = Ultra(A) \]
since every filter contains $1$. For $a,b \in A$, 
\[ s(a \wedge b) = s(a) \wedge s(b) \]
as any filter containing $a \wedge b$ also contains $a$ and $b$ and conversely, by definition of filter. And 
\[ s(a \vee b) = s(a) \vee s(b) \]
as any ultrafilter containing $a \vee b$ contains either $a$ or $b$ by the lemma above, and any filter containing either $a$ or $b$ contains $a \vee b$ by definition of filter. Finally
\[ s(\neg a) = \neg s(a) \]
as an ultrafilter contains exacly one of $a$ and $\neg a$, so the ultrafilters containing $\neg a$ are exactly those not containing $a$.

We have now proven that $s$ is a homomorphism, but still need to prove injectivity. To do this, it suffices to find, for any distinct $a,b \in A$, an ultrafilter containing $a$ but not $b$ or the other way around. Suppose without loss of generality that $a \not\le b$. Then $a$ and $\neg b$ have nonzero intersection, so ${a, \neg b}$ generates a filter. Any ultrafilter extending this filter, will contain $a$ but not $b$.
\end{proof}

\begin{theorem}
Every distributive lattice is isomorphic to a lattice of sets.
\end{theorem}

\begin{proof}
We've already proven Birkhoff's representation theorem, which says that every finite distributive lattice is isomorphic to a lattice of sets. Hence every finite distributive lattice can be lattice embedded into a BA. An application of the compactness theorem, similar to the proof of theorem \ref{mainthm}, gives that any distributive lattice can be lattice embedded into a BA. But by Stone's representation theorem, this BA is isomorphic to an algebra of sets. Hence any distributive lattice can be lattice embedded into an lattice of sets.
\end{proof}

One can also prove directly that every distributive lattice is isomorphic to a lattice of sets, and deduce Stone's representation theorem as a corollary. Such an approach is given in \cite{gratzer}.

\bibliographystyle{plain}
\bibliography{embeddings}

\end{document}